\documentclass[a4paper,12pt]{scrartcl}
\usepackage{amsmath,amssymb,latexsym,amsthm,enumerate}
\usepackage{xcolor}
\usepackage{graphicx}
\usepackage{epstopdf}

\usepackage[T1]{fontenc}
\usepackage[utf8]{inputenc}
\usepackage[english]{babel}

\newcommand*{\wh}{\widehat}
\newcommand*{\wt}{\widetilde}
\newcommand*{\ol}{\overline}

\newcommand*{\eps}{\varepsilon}

\newcommand*{\N}{\mathbb{N}}
\newcommand*{\R}{\mathbb{R}}

\newcommand*{\bbN}{\mathbb N}
\newcommand*{\bbQ}{\mathbb Q}
\newcommand*{\bbR}{\mathbb R}

\newcommand*{\cF}{\mathcal{F}}

\newcommand*{\loc}{\mathrm{loc}}

\newcommand*{\Law}{\operatorname{Law}}

\newcommand{\be}{\begin{eqnarray*}}
\newcommand{\ee}{\end{eqnarray*}}
\newcommand{\ben}{\begin{eqnarray}}
\newcommand{\een}{\end{eqnarray}}
\newcommand{\bi}{\begin{itemize}}
\newcommand{\ei}{\end{itemize}}

\newtheorem{theo}{Theorem}[section]

\newtheorem{lemma}[theo]{Lemma}
\newtheorem{propo}[theo]{Proposition}
\newtheorem{corollary}[theo]{Corollary}

\theoremstyle{definition}
\newtheorem{ex}[theo]{Example}

\newtheorem{remark}[theo]{Remark}

\newcounter{numpar}[section]

\title{Approximating exit times of continuous Markov processes}

\author{
Thomas Kruse\thanks{%
Thomas Kruse, Institute of Mathematics, University of Gie{\ss}en, Arndtstr.~2, 35392 Gießen, Germany.
\emph{Email:} thomas.kruse@math.uni-giessen.de, \emph{Phone:} +49 (0)641 9932043.}
\and Mikhail Urusov\thanks{%
Mikhail Urusov, Faculty of Mathematics, University of Duisburg-Essen, Thea-Leymann-Str.~9, 45127 Essen, Germany.
\emph{Email:} mikhail.urusov@uni-due.de, \emph{Phone:} +49 (0)201 1837428.
}}

\begin{document}

\maketitle

\begin{abstract}
The time at which a one-dimensional continuous strong Markov process attains a boundary point of its state space
is a discontinuous path functional and it is, therefore,
unclear whether the exit time can be approximated
by hitting times of approximations of the process.
We prove a functional limit theorem
for approximating weakly
both the paths of the Markov process
and its exit times.
In contrast to the functional limit theorem
in~\cite{aku2018cointossing}
for approximating the paths,
we impose a stronger assumption here.
This is essential, as we present an example showing
that the theorem extended with the convergence
of the exit times does not hold
under the assumption in~\cite{aku2018cointossing}.
However, the \mbox{EMCEL} scheme introduced
in~\cite{aku2018cointossing}
satisfies the assumption of our theorem,
and hence we have a scheme capable of
approximating both the process and its exit times
for \emph{every} one-dimensional continuous strong Markov process,
even with irregular behavior (e.g., a solution of an SDE
with irregular coefficients or a Markov process with sticky features).
Moreover, our main result can be used to check
for some other schemes whether the exit times converge.
As an application we verify that the weak Euler scheme
is capable of approximating the absorption time of the CEV diffusion
and that the scale-transformed weak Euler scheme
for a squared Bessel process is capable of approximating
the time when the squared Bessel process hits zero.


\end{abstract}

\section*{Introduction}
In this article we aim at approximating exit times of one-dimensional
regular continuous strong Markov processes
(in the sense of Section~VII.3 in~\cite{RY}
or Section~V.7 in~\cite{RogersWilliams}).
In what follows, the latter class of processes is called
\emph{general diffusions},
and the term \emph{exit time} stays for a hitting time of a boundary point of the state space.

In the introduction we consider for simplicity
a general diffusion $Y=(Y_t)_{t\in[0,\infty)}$
in natural scale with the state space $I=[0,\infty)$
and a speed measure $m$.
For the interior of the state space we use the notation
$I^\circ=(0,\infty)$.
A particular case is that $Y$ is a solution of the
Stochastic Differential Equation (SDE)
\begin{equation}\label{eq:22032019a1}
dY_t=\eta(Y_t)\,dW_t,
\end{equation}
where $W$ denotes a Brownian motion.
It is known that \eqref{eq:22032019a1} has a
(possibly reaching $0$ in finite time) unique in law
weak solution under the Engelbert-Schmidt condition
that $\eta\colon I^\circ\to\bbR$ is a non-vanishing (possibly irregular)
Borel function  such that
$1/\eta^2$ is locally integrable on $I^\circ$
(see~\cite{ES1985} or Theorem~5.5.7 in~\cite{KS}).
We extend $\eta$ to $I$ by setting $\eta(0)=0$
to enforce absorption in $0$ whenever $0$ is accessible
(whether or not $0$ is accessible depends
on the behavior of $\eta$ near~$0$).
In the case of~\eqref{eq:22032019a1}
the speed measure $m$ is absolutely continuous
with respect to the Lebesgue measure
and is given by the formula
$m(dx)=\frac2{\eta^2(x)}\,dx$.
Notice, however, that our setting is more general
than~\eqref{eq:22032019a1}
because many general diffusions
cannot be characterized in terms of an SDE
(the latter is, in particular, true for general diffusions with sticky points,
which correspond to atoms in the speed measure $m$
inside $I^\circ$,
and which gained an increased interest in recent years;
see \cite{KSS2011}, \cite{Bass2014}, \cite{ep2014},
\cite{HajriCaglarArnaudon:17}
and references therein).

The question that initiated our research is as follows.
Let $\ol h\in(0,1]$ and, for each $h\in(0,\ol h)$,
let $X^h=(X^h_t)_{t\in[0,\infty)}$
be a continuous process viewed as an approximation
of $Y=(Y_t)_{t\in[0,\infty)}$.
Assume that
\begin{equation}\label{eq:22032019a2}
X^h\xrightarrow[]{w}Y,\quad h\to0,
\end{equation}
which means that the distributions of the processes $X^h$
converge weakly to the distribution of~$Y$ (in $C([0,\infty),\bbR)$).
The question is whether we have weak convergence of the hitting times
\begin{equation}\label{eq:22032019a3}
H_0(X^h)\xrightarrow[]{w}H_0(Y),\quad h\to0,
\end{equation}
where, for a process $Z=(Z_t)_{t\in[0,\infty)}$ and $b\in\bbR$,
we use the notation
$H_b(Z)=\inf\{t\in[0,\infty):Z_t=b\}$ (with $\inf\emptyset=\infty$).
The process $X^h$ can be given by some simulation scheme
for the process $Y$ (e.g., the Euler scheme with linear interpolation between grid points for the case
when $Y$ is driven by an SDE) and
$h$ plays the role of discretization parameter.
The question is thus whether we can approximately simulate
the exit time $H_0(Y)$
having at our disposal a convergent scheme for $Y$ itself.
In the case when $0$ is an accessible boundary point for~$Y$,
this is a difficult question because the path functional $H_0$
is discontinuous and hence does not in general preserve
the weak convergence (also see Section~\ref{sec:A_not_suff}
for an example, where \eqref{eq:22032019a2} holds
but \eqref{eq:22032019a3} is violated).
In the case of the Constant Elasticity of Variance (CEV) diffusion
\begin{equation}\label{eq:22032019a4}
dY_t=Y_t^p\,dW_t
\end{equation}
with $p\in[1/2,1)$ and its Euler approximations $X^h$,
the article~\cite{ChiganskyKlebaner:12} proves the weak convergence
$H_{h^\beta}(X^h)\xrightarrow[]{w}H_0(Y)$, $h\to0$,
for any fixed $\beta\in\left(0,\frac{1/2}{1-p}\right)$,
which allows to approximately simulate the exit time
of the CEV diffusion  (to
get continuous-time processes $X^h$, the Euler scheme
is linearly interpolated between the grid points). It is, however, an open question
whether \eqref{eq:22032019a3} holds true even for the CEV diffusion $Y$ and its Euler approximations~$X^h$
(see~\cite{ChiganskyKlebaner:12} for more detail and notice that,
in contrast to the approach in~\eqref{eq:22032019a3}, the main result in~\cite{ChiganskyKlebaner:12}
requires to make the hitting boundary for the Euler scheme also depend on the discretization parameter~$h$).

The first message we would like to convey
is that the EMCEL approximation scheme $\wh X^h$,
which is well-defined for every general diffusion $Y$,
has the property~\eqref{eq:22032019a3}
for \emph{every} general diffusion $Y$.
The EMCEL scheme is introduced in~\cite{aku2018cointossing}
and is shown to be able to approximate every
general diffusion $Y$ in the sense~\eqref{eq:22032019a2}.
This scheme is recalled in Example~\ref{ex:25022019a1} below.
The second message we would like to convey
is that \eqref{eq:22032019a3} holds true
for the CEV diffusion $Y$ given in~\eqref{eq:22032019a4}
and its \emph{weak} Euler approximations $X^h$,
which resolves a variant of the open question mentioned above.

Both messages mentioned in the preceding paragraph
follow from our main result, Theorem~\ref{th:main_new},
where we consider the class of approximating schemes $X^h$
described in \eqref{eq:def_X}--\eqref{eq:13112017a1} below
and present a sufficient condition for~\eqref{eq:22032019a3}
(Condition~(B) below). In fact,
Theorem~\ref{th:main_new}
contains more than just~\eqref{eq:22032019a3}
under Condition~(B):
it is a functional limit theorem
both for the paths of the processes
and for their exit times;
see Section~\ref{sec:main_res} for more detail.
The mentioned messages are obtained as follows:
the EMCEL scheme satisfies Condition~(B)
for every general diffusion~$Y$;
the weak Euler scheme satisfies Condition~(B)
for $Y$ given by~\eqref{eq:22032019a4}.

We now discuss related literature.
The article~\cite{Ethier:79} proves weak convergence
of certain absorption times
that arise naturally in population genetics.
The question of simulating the hitting times
of squared Bessel processes
is studied in
\cite{DH:13} and~\cite{DH:17}.
Exact simulation of the first-passage time of diffusions
(in the style of~\cite{BPR:06})
is considered in~\cite{HZ:17}.
Under some regularity assumptions it is proved
in~\cite{BGG:17} that the discrete exit time
of the Euler scheme of a diffusion
converges in $L^1$ with the optimal rate $1/2$
to the continuous exit time.
For more information about the distributions
of the exit times of diffusions
see \cite{BR:16},
\cite{KaratzasRuf:16} and references therein.

\smallskip
The paper is structured as follows.
In Section~\ref{sec:as} we formally describe our setting,
the approximation schemes we use in the paper
and recall the functional limit theorem from~\cite{aku2018cointossing}
ensuring~\eqref{eq:22032019a2} under a certain Condition~(A).
Section~\ref{sec:main_res} presents and discusses the main result,
Theorem~\ref{th:main_new},
which is proved in Section~\ref{sec:proofs}.
The assumption in Theorem~\ref{th:main_new},
Condition~(B), is stronger than Condition~(A).
Section~\ref{sec:A_not_suff} contains an example
showing that Condition~(A) does not suffice for~\eqref{eq:22032019a3}.
Finally, in Section~\ref{sec:CEV} we discuss
an application of our result to the CEV diffusion~\eqref{eq:22032019a4}
and to squared Bessel processes.

\section{Approximation schemes}\label{sec:as}
Let $(\Omega, \cF, (\cF_t)_{t \ge 0}, (P_y)_{y \in I}, (Y_t)_{t \ge 0})$ be a one-dimensional continuous strong Markov process in the sense of Section~VII.3 in~\cite{RY}. We refer to this class of processes as {\itshape general diffusions} in the sequel. We assume that the state space is an open, half-open or closed interval $I \subseteq \R$. We denote by $I^\circ=(l,r)$ the interior of $I$, where $-\infty\leq l<r\leq \infty$, and we set $\ol I=[l,r]$.
Recall that by the definition we have $P_y[Y_0=y]=1$ for all $y\in I$. 
We further assume that $Y$ is regular. This means that for every $y\in I^\circ$ and $x\in I$ we have that $P_y[H_x(Y)<\infty]>0$, where $H_x(Y)=\inf\{t\geq 0: Y_t=x \}$.
Moreover, for $a<b$ in $\ol I$ we denote by $H_{a,b}(Y)$
the first exit time of $Y$ from $(a,b)$,
i.e., $H_{a,b}(Y) = H_a(Y)\wedge H_b(Y)$.
Without loss of generality
we suppose that the diffusion $Y$ is in natural scale.
If $Y$ is not in natural scale,
then there exists a strictly increasing continuous function
$s\colon I \to \R$, the so-called scale function,
such that $s(Y)$ is in natural scale.

Let $m$ be the speed measure of the Markov process $Y$
on~$I^\circ$ (see~VII.3.7 in~\cite{RY}).
Recall that for all $a<b$ in $I^\circ$ we have
\begin{equation}\label{eq:06072018a1}
0<m([a,b])<\infty.
\end{equation}
We assume that if a boundary point
(that is, $l$ or $r$)
is accessible, then it is absorbing.
For our purposes,
this assumption is without loss of generality,
as for any general diffusion $Y$ --- 
possibly with reflecting boundary points ---
the stopped process $Y^{H_{l,r}(Y)}$ is a general diffusion with
absorbing boundary points and has the same exit times as $Y$.

\begin{ex}[Driftless SDE with possibly irregular diffusion coefficient]\label{ex:sde}
Consider the case,
where inside $I^\circ$ the process $Y$ is driven by the SDE
\begin{equation}\label{eq:27092018a1}
dY_t=\eta(Y_t)\,dW_t,
\end{equation}
where $\eta\colon I^\circ\to\bbR$ is a Borel function
satisfying the Engelbert-Schmidt conditions
\begin{gather}
\eta(x)\ne0\;\;\forall x\in I^\circ,
\label{eq:27092018a2}\\[1mm]
\eta^{-2}\in L^1_{\loc}(I^\circ)
\label{eq:27092018a3}
\end{gather}
($L^1_{\loc}(I^\circ)$ denotes the set of Borel functions
locally integrable on~$I^\circ$).
Under \eqref{eq:27092018a2}--\eqref{eq:27092018a3}
SDE~\eqref{eq:27092018a1}
has a unique in law
(possibly exiting $I^\circ$ in finite time)
weak solution;
see~\cite{ES1985} or Theorem~5.5.7 in~\cite{KS}.
We make the convention that $Y$ remains constant
after reaching $l$ or $r$ in finite time,
which makes the boundary points absorbing whenever accessible.
This is a particular case of our setting,
where the speed measure of $Y$
on $I^\circ$ is given by the formula
\begin{equation}\label{eq:speed_measure_sde}
m(dx)=\frac 2{\eta^2(x)}\,dx.
\end{equation}
\end{ex}

We now describe the approximation schemes considered in this paper.
Let $\ol h \in (0,1]$ and suppose that for every $h \in (0, \ol h)$ we are given a Borel function $a_{h}\colon \ol I \to [0,\infty)$ such that $a_h(l)=a_h(r)=0$ and for all $y\in I^\circ$ we have $y\pm a_h(y)\in I$. We refer to each function $a_h$ as a \emph{scale factor}.
We next construct a family of Markov chains associated to the family of scale factors $(a_h)_{h\in (0, \ol h)}$. To this end
we fix a starting point $y \in I^\circ$ of $Y$.
Let $(\xi_k)_{k \in \bbN}$ be an iid sequence of random variables,
on a probability space with a measure $P$,
satisfying $P(\xi_k = \pm 1) = \frac12$. 
We denote by $(X^h_{kh})_{k \in \bbN_0}$ the Markov chain defined by
\begin{equation}\label{eq:def_X}
X^h_0 = y
\quad\text{and}\quad
X^h_{(k+1)h} = X^h_{kh}+ a_h(X^h_{kh}) \xi_{k+1}, \quad \text{ for } k \in \bbN_0. 
\end{equation}
We extend $(X^h_{kh})_{k \in \bbN_0}$ to a continuous-time process by linear interpolation, i.e., for all $t\in[0,\infty)$, we set
\begin{equation}\label{eq:13112017a1}
X^h_t = X^h_{\lfloor t/h \rfloor h} + (t/h - \lfloor t/h \rfloor) (X^h_{(\lfloor t/h \rfloor +1)h} - X^h_{\lfloor t/h \rfloor h}). 
\end{equation}
To highlight the dependence of $X^h=(X^h_t)_{t\in[0,\infty)}$ on the starting point $y\in I^\circ$ we also sometimes write~$X^{h,y}$.

Next we recall the main result in~\cite{aku2018cointossing},
which allows to approximate $Y$ with such Markov chains.
Here and in the sequel
we equip $C([0,\infty),\bbR)$ with the topology
of uniform convergence on compact intervals,
which is generated, e.g., by the metric
$$
\rho(x,y)=\sum_{n=1}^\infty 2^{-n}
\left(\|x-y\|_{C[0,n]}\wedge1\right),
\quad x,y\in C([0,\infty),\bbR),
$$
where $\|\cdot\|_{C[0,n]}$ denotes the sup norm
on $C([0,n],\bbR)$.
Moreover, we need the following condition.

\paragraph{Condition~(A)}
For all compact subsets $K$ of $I^\circ$ it holds that
\begin{equation}
\sup_{y\in K }  \left|
\frac{1}{2}\int_{(y-a_h(y),y+a_h(y))} (a_h(y)-|u-y|)\,m(du)
-h
\right|
\in o(h), \quad h \to 0.
\end{equation}

\begin{theo}[Theorem~1.1 in~\cite{aku2018cointossing}]\label{thm_main_coin}
Assume that Condition~(A) is satisfied. Then, for any $y\in I^\circ$,
the distributions of the processes
$(X^{h,y}_{t})_{t \in [0,\infty)}$ under $P$
converge weakly to the distribution of
$(Y_t)_{t \in [0,\infty)}$ under $P_{y}$, as $h \to 0$;
i.e., for every bounded and continuous
functional
$F\colon C([0,\infty),\bbR)\to\bbR$,
it holds that
\ben\label{eq:13112017a2}
E[F(X^{h,y})]\to E_y[F(Y)], \quad h\to 0.
\een
\end{theo}

\section{Main result}\label{sec:main_res}
We first introduce an auxiliary subset of $I^\circ$.
If $l> -\infty$, we define, for all $h\in(0,\ol h)$,
\begin{align*}
l_h := l+ \inf\left\{ a \in \left(0,\frac{r-l}2\right]:
a<\infty
\;\;\text{and}\;\;
\frac12\int_{(l, l+2a)} (a - |u-(l+a)|) m(du) \ge h \right\}, 
\end{align*}
where we use the convention $\inf \emptyset = \infty$.
If $l = -\infty$, we set $l_h = -\infty$. 
Similarly, if $r < \infty$, then we define, for all $h\in(0,\ol h)$,
\begin{align*}
r_h := r- \inf\left\{ a \in \left(0,\frac{r-l}2\right]:
a<\infty
\;\;\text{and}\;\;
\frac12\int_{(r-2a, r)} (a - |u-(r-a)|) m(du) \ge h \right\}.
\end{align*}
If $r = \infty$, we set $r_h = \infty$. 
We refer to Section 3 in \cite{aku2018cointossing} for a discussion of the extended real numbers $r_h$ and $l_h$.
In particular, it holds that $l$ is inaccessible if and only if $l_h = l$ for all $h\in (0,\ol h)$. Similarly, $r$ is inaccessible if and only if $r_h = r$ for all $h\in (0,\ol h)$.
If $l$ or $r$ are accessible it holds that $l_h\searrow l$ or $r_h\nearrow r$, respectively, as $h\searrow 0$.

Now the auxiliary subset is defined by
\begin{equation}\label{eq:22022019a2}
I_h = (l_h,r_h)\cup\left\{y\in I^\circ \colon y\pm a_h(y)\in I^\circ\right\}, \quad h\in (0,\ol h).
\end{equation}
We observe that $(l_h,r_h)\subseteq I_h\subseteq I^\circ$
and notice that $I_h$ is, in general, not connected,
as it depends on the behavior of the scale factor $a_h$,
which is only a Borel function $a_h\colon\ol I\to[0,\infty)$
with the properties $a_h(l)=a_h(r)=0$ and $y\pm a_h(y)\in I$
for all $y\in I^\circ$.
But, as $I_h\supseteq (l_h,r_h)$, ``gaps''
can appear only close to accessible boundaries.

In order to formulate the main result we need to discuss 
the following conditions.

\paragraph{Condition~(B)}
(i) There exist $B\in [1,\infty)$, $\gamma\in (\frac{1}{2},1]$ 
and a function $\alpha \colon (0,\ol h) \to [0,\infty)$ with $\lim_{h\searrow 0} \frac{\alpha(h)}{h}=1$ such that for all $h\in (0,\ol h)$ and $y\in I_h$ 
\begin{equation}\label{eq:01042019a1}
\alpha (h) \le \frac{1}{2}\int_{(y-a_h(y),y+a_h(y))} (a_h(y)-|u-y|)\,m(du)\le Bh^\gamma.
\end{equation}

(ii) For every compact subset $K$ of $I^\circ$ there exists a function
$\beta_K \colon (0,\ol h) \to [0,\infty)$ with $\lim_{h \searrow 0} \frac{\beta_K(h)}{h}=1$
such that for all $h\in (0,\ol h)$ and $y\in K$ 
\begin{equation}\label{eq:04042019b1}
\frac{1}{2}\int_{(y-a_h(y),y+a_h(y))} (a_h(y)-|u-y|)\,m(du)
\le \beta_K(h).
\end{equation}

We avoid introducing ``Condition~(C)'' in this paper
to escape a collision with an\linebreak
unrelated Condition~(C)
in the article~\cite{aku2019wasserstein},
which studies convergence rates of the EMCEL and related schemes,
and proceed with

\paragraph{Condition~(D)}
It holds that
\begin{equation}\label{eq:23032019a3}
\sup_{y\in I_h}  \left|
\frac{1}{2}\int_{(y-a_h(y),y+a_h(y))} (a_h(y)-|u-y|)\,m(du)
-h
\right|
\in o(h), \quad h \to 0.
\end{equation}

\smallskip
It is easy to see that Condition~(B) is stronger than Condition~(A), while Condition~(D) 
is stronger than Condition~(B). We prove our main result, Theorem~\ref{th:main_new} below, under Condition~(B).
In Section~\ref{sec:A_not_suff} we provide an example, where Condition~(A) is satisfied but the claim of Theorem~\ref{th:main_new} does not hold true. Condition~(D) can be viewed as a symmetric sufficient condition for the claim of Theorem~\ref{th:main_new}, and it is, in fact, enough to deliver the first message mentioned in the introduction.
On the contrary, to provide an answer to the open question about the convergence of the exit times for the (weak) Euler approximations of the CEV diffusion, we do need the full strength of Theorem~\ref{th:main_new} under Condition~(B).

We recall that $X^{h,y}$ and $Y$
denote the whole continuous-time processes:
$X^{h,y}=(X^{h,y}_t)_{t\in[0,\infty)}$,
$Y=(Y_t)_{t\in[0,\infty)}$.
To formulate the main result of this article,
Theorem~\ref{th:main_new}, we equip
$[0,\infty]$ with the metric
$$
d(s,t)=\left|\frac{s}{1+s}-\frac{t}{1+t} \right|,\quad s,t\in[0,\infty],
$$
and we use the standard product topology on product spaces,
which is generated, e.g., by the metric on the product space
defined as sum of the distances between the components
(also recall the paragraph preceding Condition (A)).

\begin{theo}\label{th:main_new}
Assume that Condition~(B) is satisfied. Then, for any $y\in I^\circ$,
the distributions of the random elements
$(H_l(X^{h,y}),H_r(X^{h,y}),X^{h,y})$ under $P$ converge weakly
to the distribution of $(H_l(Y),H_r(Y),Y)$ under $P_y$, as $h\to 0$; i.e.,
for every bounded and continuous
functional
$F\colon [0,\infty]^2\times C([0,\infty),\bbR)\to \R$, it holds that
\begin{equation}\label{eq:23032019a1}
E[F(H_l(X^{h,y}),H_r(X^{h,y}),X^{h,y})]\to E_y[F(H_l(Y),H_r(Y),Y)], \quad h\to 0.
\end{equation}
\end{theo}

As usual, for the weak convergence~\eqref{eq:23032019a1}
we use the shorthand notation
\begin{equation}\label{eq:23032019a2}
(H_l(X^{h,y}),H_r(X^{h,y}),X^{h,y})
\xrightarrow[]{w}
(H_l(Y),H_r(Y),Y),\quad h\to0.
\end{equation}
We remark that \eqref{eq:23032019a2}
immediately implies the weak convergence of the marginals.
In particular,
Theorem~\ref{th:main_new} establishes more than Theorem~\ref{thm_main_coin},
but this is achieved under the stronger Condition~(B)
and cannot be achieved under Condition~(A)
(see Section~\ref{sec:A_not_suff}).
Furthermore, \eqref{eq:23032019a2}
immediately implies the weak convergence
$$
(H_{l,r}(X^{h,y}),X^{h,y})
\xrightarrow[]{w}
(H_{l,r}(Y),Y),\quad h\to0,
$$
as $(s,t,x)\mapsto(s\wedge t,x)$ is a continuous function
$[0,\infty]^2\times C([0,\infty),\bbR)\to[0,\infty]\times C([0,\infty),\bbR)$.

It is important to note that for every speed measure $m$ there exists a family of scale factors such that Condition~(D), and hence Condition~(B), is satisfied. 
Consequently, Theorem~\ref{th:main_new} entails that the exit times of
\emph{every}
general diffusion $Y$ can be approximated
with the help of Markov chains of the form~\eqref{eq:def_X}.
These scale factors are provided in the next example.

\begin{ex}[EMCEL approximations]\label{ex:25022019a1}
Let $h\in(0,\ol h)$.
The EMCEL$(h)$ scale factor $\wh a_h$ is defined by
$\wh a_h(l)=\wh a_h(r)=0$
and, for all $y\in I^\circ$,
\begin{equation}\label{sf absorbing case}
\wh a_h(y) = \sup\left\{a \ge 0: y\pm a \in I \text{ and } \frac{1}{2}\int_{(y-a,y+a)} (a-|z-y|)\,m(dz) \le h\right\}.
\end{equation}
The associated process
defined in \eqref{eq:def_X}--\eqref{eq:13112017a1}
is denoted by $\wh X^h$ and referred to as
Embeddable Markov Chain with Expected time Lag~$h$
(we write shortly $\wh X^h\in\text{EMCEL}(h)$). The
whole family $(\wh X^h)_{h\in (0,\ol h)}$ is referred to as the
\emph{EMCEL approximation scheme.}
Alternatively, we simply say
\emph{EMCEL approximations.}

We now explain in more detail what is included in
Definition~\eqref{sf absorbing case}.
For all $y\in I^\circ$, we define $a_I(y)=\min\{y-l,r-y\}$ ($\in(0,\infty]$)
and notice that, for $a\ge0$, it holds
$y\pm a\in I^\circ$ if and only if $a<a_I(y)$.
Fix $y\in I^\circ$.
It follow from~\eqref{eq:06072018a1} that
$\int_{(y-a,y+a)} (a-|z-y|)\,m(dz)<\infty$ whenever $a\in[0,a_I(y))$.
Therefore, the function
$$
a\mapsto\int_{(y-a,y+a)}(a-|z-y|)\,m(dz)\equiv\int_I (a-|z-y|)^+\,m(dz)
$$
is strictly increasing and continuous on $[0,a_I(y))$ (by the dominated convergence theorem).
The definitions of $l_h$ and $r_h$ yield that, 
for $y\in(l_h,r_h)$, the number $\wh a_h(y)$ is a unique positive root
of the equation (in~$a$)
\begin{equation}\label{eq:22022019a3}
\frac{1}{2}\int_{(y-a,y+a)} (a-|z-y|)\,m(dz) = h,
\end{equation}
while, for $y\in(l,l_h]$ (resp., $y\in[r_h,r)$),
$\wh a_h(y)$ is chosen to satisfy
\begin{equation}\label{eq:22022019a4}
y-\wh a_h(y)=l\qquad\text{(resp., }y+\wh a_h(y)=r).
\end{equation}
It follows from~\eqref{eq:22022019a4} that
the set $I_h$ of~\eqref{eq:22022019a2}
corresponding to the EMCEL$(h)$ scale factor $\wh a_h$
(and, naturally, denoted by $\wh I_h$) is simply
$$
\wh I_h=(l_h,r_h).
$$
This yields that the left-hand side in~\eqref{eq:23032019a3}
vanishes for the EMCEL approximations and,
therefore, for this scheme Condition~(D), and hence Condition~(B),
is satisfied.
\end{ex}

\begin{remark}
As shown in Example~\ref{ex:25022019a1} the EMCEL scale factors always satisfy Condition~(D). 
However, Equation~\eqref{eq:22022019a3}
defining these scale factors
can rarely be solved in closed form.
Therefore, in practice we usually need to solve 
\eqref{eq:22022019a3} approximately. 
Condition~(D) dictates that we 
need to solve \eqref{eq:22022019a3} with an error of order $o(h)$ 
uniformly in $y\in I_h$ in order to
guarantee convergence of the associated 
exit times.
Condition~(B) is a certain asymmetric
weakening of the required precision. 

Moreover, we note that Theorem~\ref{th:main_new} is not only applicable to perturbations of the EMCEL approximation but also can be used to
derive convergence results for exit times in other approximation methods, e.g., for the weak Euler scheme.
This is illustrated in Section~\ref{sec:CEV}.
\end{remark}

\begin{remark}\label{rem:01042019a1}
One might wonder why we consider only the question
of convergence of the exit times from $I^\circ$
rather than considering the task of approximating
$H_b(Y)$ for any $b\in I$.
The answer is that this question is interesting
(and difficult --- see the paragraph containing
\eqref{eq:22032019a2}, \eqref{eq:22032019a3}
and~\eqref{eq:22032019a4}
in the introduction)
only for $b\in\{l,r\}$.
For $b\in I^\circ$, the path functional
$H_b\colon C([0,\infty),\bbR)\to[0,\infty]$
is $P_y$-a.s.\ continuous for any $y\in I^\circ$.
Indeed, the functional $H_b$, for 
$b\in  I^\circ$, is only discontinuous at paths $x \in C([0,\infty),\bbR)$
that at some point in time touch the level $b$ but do not cross it, i.e.,
at paths $x$ 
with local maximum or minimum value $b$.
Formally,
in the case $y\ge b\in I^\circ$ (resp., $y\le b\in I^\circ$)
the points of discontinuity of $H_b$
intersected with
$\{x\in C([0,\infty),\bbR):x(0)=y\}$
are contained in the set
$$
\bigcup_{q\in(0,\infty)\cap\bbQ}
\{x\in C([0,\infty),\bbR):\inf_{t\in[0,q]}x(t)=b<x(q)\}
$$
(resp., in the set given by the similar formula,
where ``$\inf$'' is replaced by ``$\sup$''
and ``$<$'' by ``$>$''),
while the latter set has $P_y$-measure zero,
which follows from the strong Markov property of $Y$
and the oscillating behavior of $Y$ at time zero.
More precisely, we refer to the property
$$
P_b(H^+_b(Y)=0)=P_b(H^-_b(Y)=0)=1,
$$
where
$$
H^+_b(Y)=\inf\{t\ge0:Y_t>b\}
\quad\text{and}\quad
H^-_b(Y)=\inf\{t\ge0:Y_t<b\},
$$
which follows from the construction of $Y$ as a time-changed Brownian motion,
see Theorem~V.47.1 in~\cite{RogersWilliams}.
Therefore, for any $b\in I^\circ$, we obtain
$H_b(X^{h,y})\xrightarrow[]{w}H_b(Y)$, $h\to0$,
for any starting point $y\in I^\circ$
under Condition~(A) just as a corollary
of Theorem~\ref{thm_main_coin}.
Moreover, the same reasoning
immediately leads to the following extension of
Theorem~\ref{th:main_new}:

Assume that Condition~(B) is satisfied.
Let $n\in\bbN_0$ and $b_1,\ldots,b_n\in I^\circ$.
Then, for any $y\in I^\circ$,
the distributions of the random elements
$$
(H_l(X^{h,y}),H_r(X^{h,y}),
H_{b_1}(X^{h,y}),\ldots,H_{b_n}(X^{h,y}),
X^{h,y})
$$
under $P$
converge weakly to the distribution of
$$
(H_l(Y),H_r(Y),
H_{b_1}(Y),\ldots,H_{b_n}(Y),
Y)
$$
under $P_y$, as $h\to 0$.
\end{remark}

\section{Proof of Theorem~\protect\ref{th:main_new}}\label{sec:proofs}
The exit times considered in Theorem~\ref{th:main_new} may attain the value $\infty$ with positive probability. This is why we introduced in the text preceding Theorem~\ref{th:main_new} the metric $d$ on the nonnegative real line including $\infty$. The next result shows that it suffices to verify convergence in probability on compact time intervals.

\begin{lemma}\label{lem:finite_conv_prob}
Let $f\colon [0, \infty] \to [0,1]$ be an increasing bijection and let $d\colon [0,\infty]^2 \to [0,1]$ be the metric satisfying $d(s,t)=|f(s)-f(t)|$, $s,t\in [0,\infty]$. Let $(\Omega, \mathcal F, P)$ be a probability space, let $\xi\colon \Omega \to [0,\infty]$ be a random variable and let $\xi_n\colon \Omega \to [0,\infty]$,  $n\in \N$, be a sequence of random variables such that for all $T\in [0,\infty)$ the sequence of $[0,T]$-valued random variables
$(\xi_n\wedge T)_{n\in \N}$ converges to $\xi \wedge T$ in probability (with respect to the metric induced by the absolute value). Then the sequence $(\xi_n)_{n\in \N}$ converges to $\xi$ in probability with respect to the metric $d$ on $[0,\infty]$.
\end{lemma}
\begin{proof}
Throughout the proof fix $\varepsilon\in (0,\infty)$. We need to show that $P(d(\xi_n,\xi)>\varepsilon)\to 0$ as $n\to \infty$. For all $n\in \N$ introduce the set
$$
A_n=\left\{f(\xi_n\vee \xi) \le 1-\frac \eps 2, d(\xi_n,\xi)>\eps\right\}.
$$
Since $f$ is continuous on $[0,\infty)$ it is uniformly continuous on $[0,f^{-1}(1-\frac \eps 2)]$, i.e., there exists $\delta\in (0,\infty)$ such that for all $s,t \in [0,f^{-1}(1-\frac \eps 2)]$
with $|s-t|\le \delta$ it holds that $d(s,t)=|f(s)-f(t)|\le \eps$.
This together with the assumption that 
$(\xi_n\wedge ( f^{-1}\left(1-\frac \eps 2\right)))_{n\in \N}$ converges to $\xi \wedge (f^{-1}\left(1-\frac \eps 2\right))$ in probability
 implies
\begin{equation}\label{eq:conv_A_n}
P(A_n)\le P\left((\xi_n\vee \xi) \le f^{-1}\left(1-\frac \eps 2\right), |\xi_n-\xi|>\delta\right)\to 0, 
\quad n\to \infty.
\end{equation}
Next note that for all $n\in \N$
\begin{equation}
\begin{split}
\left\{d(\xi_n,\xi)>\eps\right\}\setminus A_n&=\left\{f(\xi_n\vee \xi) > 1-\frac \eps 2, d(\xi_n,\xi)>\eps\right\}\\
&=\left\{f(\xi_n\vee \xi) > 1-\frac \eps 2, d(\xi_n,\xi)>\eps, f(\xi_n\wedge \xi) \le 1- \eps \right\}\\
&\subseteq 
\left\{ d\left (\xi_n\wedge\left(f^{-1}\left(1-\frac \eps 2\right)\right),\xi\wedge\left(f^{-1}\left(1-\frac \eps 2\right)\right)\right)\ge\frac{\eps}{2} \right\}.
\end{split}
\end{equation}
Using similar arguments as in \eqref{eq:conv_A_n} we obtain that
$P\left( \left\{d(\xi_n,\xi)>\eps\right\}\setminus A_n\right)\to 0$ as $n\to \infty$. 
Combining this with \eqref{eq:conv_A_n} we obtain that $P\left( \left\{d(\xi_n,\xi)>\eps\right\}\right)\to 0$ as $n\to \infty$. This completes the proof.
\end{proof}

We now proceed with the proof of Theorem~\ref{th:main_new}.
It follows from the results in~\cite{aku2018cointossing} that 
the discrete-time Markov chain $(X^{h,y}_{kh})_{k\in \N_0}$
can be embedded into $Y$ with a sequence of stopping times.
More precisely, Proposition~3.4 in \cite{aku2018cointossing} ensures that
for all $h\in (0,\ol h)$ and $y\in I^\circ$ there exists a sequence of stopping times $(\tau^h_k)_{k \in \N_0}$ such that
\begin{equation}\label{eq:1802a7}
\Law_{P_y}
\left(Y_{\tau^h_k}; k\in \N_0 \right)
=\Law_P
\left(X^{h,y}_{kh}; k\in \N_0\right).
\end{equation}
In what follows $(\tau^h_k)_{k\in\bbN_0}$ denotes the sequence
of stopping times from Proposition~3.4 in~\cite{aku2018cointossing}.
For every $h\in (0,\ol h)$ let $(Y^h_{kh})_{k\in \N_0}$ 
be the discrete-time process satisfying $ Y^h_{kh}=Y_{\tau^h_k}$, $k\in \N_0$.
Similarly to \eqref{eq:13112017a1}, we extend
$( Y^h_{kh})_{k\in \N_0}$ to a continuous-time
process $( Y^{h}_t)_{t\in [0,\infty)}$ by linear interpolation, i.e.,
for all $t\in[0,\infty)$, we set
\begin{equation}\label{eq:lin_int_pol}
 Y^h_t =  Y^h_{\lfloor t/h \rfloor h} + (t/h - \lfloor t/h \rfloor) ( Y^h_{(\lfloor t/h \rfloor +1)h} -  Y^h_{\lfloor t/h \rfloor h}). 
\end{equation}
Then it follows from \eqref{eq:1802a7} and \eqref{eq:13112017a1}
that
for all $h\in (0,\ol h)$ and $y\in I^\circ$
\begin{equation}\label{eq:1802a5}
\Law_{P_y}
\left( Y^h_{t}; t\in [0,\infty) \right)
=\Law_P
\left(X^{h,y}_{t}; t\in [0,\infty)\right).
\end{equation}
Therefore, in order to establish~\eqref{eq:23032019a1}
it suffices to show for all $y\in I^\circ$ that
\begin{equation}\label{eq:23032019a4}
(H_l( Y^h),H_r( Y^h),Y^h)
\xrightarrow[]{(P_y,\langle dd\rho\rangle)}
(H_l(Y),H_r(Y),Y),\quad h\to0,
\end{equation}
where the notation
$\xrightarrow[]{(P_y,\langle dd\rho\rangle)}$
stands for the convergence in probability $P_y$
under the metric $\langle dd\rho\rangle$
on $[0,\infty]^2\times C([0,\infty),\bbR)$ given by the formula
$$
\langle dd\rho\rangle\left((s,t,x),(s',t',x')\right)
=d(s,s')+d(t,t')+\rho(x,x'),
$$
where the metrics $\rho$ and $d$ are defined in
the texts preceding Condition~(A) and Theorem~\ref{th:main_new}.
Clearly, it is enough to show the convergence
of the three marginals separately.
We first notice that for all $y\in I^\circ$
$$
Y^h\xrightarrow[]{(P_y,\rho)}Y,\quad h\to0,
$$
as it is nothing else but convergence in probability $P_y$
uniformly on compact intervals,
and it is established under Condition~(A) in the proof of
Theorem~1.1 in~\cite{aku2018cointossing}.
Concerning the convergence of the remaining marginals
in~\eqref{eq:23032019a4},
by symmetry it is enough only to show that for all $y\in I^\circ$
$$
H_l( Y^h)\xrightarrow[]{(P_y,d)}H_l(Y),\quad h\to0.
$$
By Lemma~\ref{lem:finite_conv_prob}
(applied with $f(t)=\frac{t}{1+t}$, $t\in [0,\infty]$)
it is sufficient to prove
for any finite time horizon $T\in[0,\infty)$ and 
for all $y\in I^\circ$ that
\begin{equation}\label{eq:1802a6}
H_l( Y^h)\wedge T
\xrightarrow[]{P_y}
H_l(Y)\wedge T,\quad h\to0
\end{equation}
(in~\eqref{eq:1802a6} we use the standard Euclidean distance,
hence the simpler notation
$\xrightarrow[]{P_y}$,
as the random variables now take finite values).

We verify \eqref{eq:1802a6} in a two-step procedure: We first show 
for every $\eps\in (0,\infty)$
that $H_l(Y)\wedge T$ is dominated by 
$ (H_l(Y^h)\wedge T)+\eps$ on a set with 
probability arbitrarily close to $1$
as $h\to 0$ (Proposition~\ref{prop:bound_above}).
The reverse direction is established in Proposition~\ref{prop:bound_below} with
the help of Lemma~\ref{lem:ub_temp_diff_prob} and Lemma~\ref{lem:ub_time_lag}.
To formulate Proposition~\ref{prop:bound_above} it suffices to impose Condition~(A). Recall that Condition~(B) implies Condition~(A).

\begin{propo}\label{prop:bound_above}
Suppose that Condition (A) is satisfied. Then
 for all $y\in I^\circ$, $\eps\in (0,\infty)$ and $T\in [0,\infty)$ it holds that
 \begin{equation}\label{eq:ub_diff_ex_times2}
 P_y\left [(H_l(Y)\wedge T)\le (H_l(Y^h)\wedge T)+\eps\right]\to 1, \quad h\to 0.
\end{equation} 
\end{propo} 
\begin{proof}
Throughout the proof fix
$y\in I^\circ$, $T\in [0,\infty)$, $\eps\in (0,\infty)$ and $\delta\in (0,\infty)$.

First choose $\ol l\in I^\circ$ such that
$$
P_y\left [A\right]>1-\frac{\delta}{2}, \quad \text{where } A=\left\{(H_l(Y)\wedge T)-(H_{\ol l}(Y)\wedge T)\le \frac{\eps}{2}\right\}.
$$
Next, note that the functional $H_{\ol l}\wedge T\colon C([0,T],I)\to [0,T]$ is $P_y$-a.s.\ continuous with respect to the sup norm on $C([0,T],I)$. Moreover, it follows from the proof of Theorem~1.1 in \cite{aku2018cointossing} (see, in particular, (40) therein) that under Condition (A) we have 
$
\|Y^h-Y\|_{C[0,T]}\xrightarrow[]{P_y}0
$
for $h\to 0$.
This implies that there exists $h_0\in (0,\ol h)$ such that for 
all $h\in (0,h_0)$ we have
$$
P_y\left [B_h\right]>1-\frac{\delta}{2}, \quad \text{where } B_h=\left\{(H_{\ol l}(Y)\wedge T)-(H_{\ol l}(Y^h)\wedge T)\le \frac{\eps}{2}\right\}.
$$
For all $h\in (0,h_0)$ we have 
$(H_{\ol l}(Y^h)\wedge T)
\le (H_l(Y^h)\wedge T)$ and hence on the event $A\cap B_h$ 
\begin{equation*}
\begin{split}
(H_l(Y)\wedge T)-(H_l(Y^h)\wedge T)&\le (H_l(Y)\wedge T)-(H_{\ol l}(Y)\wedge T)+(H_{\ol l}(Y)\wedge T)-(H_{\ol l}(Y^h)\wedge T)\\
&\le \eps.
\end{split}
\end{equation*}
Consequently, we have for all $h\in (0,h_0)$ that
$$
P_y\left [(H_l(Y)\wedge T)-(H_l(Y^h)\wedge T)\le \eps\right]\ge 
P_y\left [A\cap B_h\right]> 1-\delta.
$$
This completes the proof.
\end{proof}

Next, note that Corollary~3.3 in \cite{aku2018cointossing} shows that
for all $h\in (0,1)$ and $y\in I^\circ$ we have
\begin{align*}
E_y[H_{y-a_h(y), y + a_h(y)}(Y)]
=\frac{1}{2}\int_{(y-a_h(y),y+a_h(y))} (a_h(y)-|u-y|)\,m(du).
\end{align*}
Therefore, Condition~(B) ensures that for all
$h\in (0,\ol h)$ and $y\in I_h$ we have that
$
E_y[H_{y-a_h(y), y + a_h(y)}(Y)]\le Bh^\gamma.
$
Concerning the complement of $I_h$, the definition of $I_h$ shows that 
for all $y\in I\setminus I_h$ we have $\frac{1}{2}\int_{(y-a_h(y),y+a_h(y))} (a_h(y)-|u-y|)\,m(du)\le h$. 
Consequently, it follows
for all $h\in (0,\ol h)$ and $y\in I$ that
(recall that $\ol h\le 1$ and $B\ge1$)
\begin{equation}\label{eq:cons_cond_d}
E_y[H_{y-a_h(y), y + a_h(y)}(Y)]\le Bh^\gamma.
\end{equation}
We recall that $(\tau^h_k)_{k\in\bbN_0}$ denotes the sequence
of stopping times from Proposition~3.4 in~\cite{aku2018cointossing}
(see, in particular,~\eqref{eq:1802a7}).

\begin{lemma}\label{lem:ub_temp_diff_prob}
Suppose that Condition~(B) is satisfied. Then
 for all $y\in I^\circ$, $\eps\in (0,\infty)$ and $T\in [0,\infty)$ it holds that
 \begin{equation}\label{eq:ub_temp_diff_prob}
 P_y\left [\sup_{k\in \{1,\ldots,\lfloor T/h\rfloor\}}(k(h\wedge \alpha(h))-\tau^h_k)\le \eps\right]\to 1, \quad h\to 0.
\end{equation} 
 \end{lemma}
 \begin{proof}
 Throughout the proof fix $y\in I^\circ$ and $T\in [0,\infty)$.
For all $k\in \N_0$ we define $\mathcal G_k=\mathcal F_{\tau^h_k}$.
By Proposition~3.4 in \cite{aku2018cointossing} and Condition~(B) we have for all $n\in \N$ and $h\in (0,\ol h)$ on the event $\{Y_{\tau^h_{n-1}} \in I_h\}$ 
 $$
E_y[\tau^h_n-\tau^h_{n-1}|\mathcal G_{n-1}]=
\frac{1}{2}\int_{(Y_{\tau^h_{n-1}}-a_h(Y_{\tau^h_{n-1}}),Y_{\tau^h_{n-1}}+a_h(Y_{\tau^h_{n-1}}))} (a_h(Y_{\tau^h_{n-1}})-|u-Y_{\tau^h_{n-1}}|)\,m(du)
\ge \alpha(h)
.
$$
Again by Proposition~3.4 in \cite{aku2018cointossing} on the event $\{Y_{\tau^h_{n-1}} \notin I_h\}$ we have $E_y[\tau^h_n-\tau^h_{n-1}|\mathcal G_{n-1}]= h$.
This implies for all $h\in (0,\ol h)$ and $k\in \N$ that
$$
k(h\wedge \alpha(h))=\sum_{n=1}^k(h\wedge \alpha(h))\le \sum_{n=1}^k E_y[\tau^h_n-\tau^h_{n-1}|\mathcal G_{n-1}]
$$
and hence
\begin{equation}\label{eq:ub_temp_diff}
\sup_{k\in \{1,\ldots, \lfloor T/h\rfloor\}}(k(h\wedge \alpha(h))-\tau^h_k)\le 
\sup_{k\in \{1,\ldots, \lfloor T/h\rfloor\}}\left (\left(\sum_{n=1}^k E_y[\tau^h_n-\tau^h_{n-1}|\mathcal G_{n-1}]\right)-\tau^h_k\right).
\end{equation}
Next, \eqref{eq:cons_cond_d} ensures that
\begin{equation*}
\begin{split}
\frac{\sup_{z\in I}\left( E_z[H_{z-a_h(z), z + a_h(z)}(Y)]\right)}{\sqrt{h}}
\le 
Bh^{\gamma-\frac{1}{2}}
\to0,\quad h\to0.
\end{split}
\end{equation*}
Therefore, Proposition~5.2 in \cite{aku2018cointossing}
ensures that
$$
E_y\left[\sup_{k\in \{1,\ldots, \lfloor T/h\rfloor \}}\left|\tau^h_k-\sum_{n=1}^k E[\tau^h_n-\tau^h_{n-1}|\mathcal G_{n-1}]\right|\right]\to 0, \quad h\to 0.
$$
Combining this with \eqref{eq:ub_temp_diff} proves \eqref{eq:ub_temp_diff_prob} and completes the proof.
 \end{proof}

\begin{lemma}\label{lem:ub_time_lag}
Suppose that Condition~(B) is satisfied. Then
 for all $y\in I^\circ$, $\eps\in (0,\infty)$ and $T\in [0,\infty)$ it holds that
 \begin{equation}\label{eq:ub_time_lag}
 P_y\left [\sup_{k\in \{1,\ldots,\lfloor T/h\rfloor\}}(\tau^h_k-\tau^h_{k-1})\le \eps\right]\to 1, \quad h\to 0.
\end{equation} 
 \end{lemma}
 
 \begin{proof}
 Throughout the proof fix $y\in I^\circ$, $T\in [0,\infty)$, $\eps\in (0,\infty)$, $h\in (0,\ol h)$ and let $N=\lfloor T/h\rfloor$. By Markov's inequality it holds that
 $$
  P_y\left [\sup_{k\in \{1,\ldots,N\}}(\tau^h_k-\tau^h_{k-1})> \eps\right]
  \le
\frac{1}{\eps^2}
E_y\left[
\sup_{k\in \{1,\ldots,N\}}(\tau^h_k-\tau^h_{k-1})^2
\right]
  \le 
  \frac{1}{\eps^2} \sum_{k=1}^NE_y\left [(\tau^h_k-\tau^h_{k-1})^2\right].
  $$
  It follows from Lemma~5.1 in \cite{aku2018cointossing} that
  $$
  P_y\left [\sup_{k\in \{1,\ldots,N\}}(\tau^h_k-\tau^h_{k-1})> \eps\right]
  \le 
  \frac{1}{\eps^2} \sum_{k=1}^N
  \left[
  2h+
2^{3/2} \sup_{z\in I}\left( E_z[H_{z-a_h(z), z + a_h(z)}(Y)]\right)
\right]^2.
  $$
  Then \eqref{eq:cons_cond_d} shows that
  \begin{equation}
  \begin{split}
  P_y\left [\sup_{k\in \{1,\ldots,N\}}(\tau^h_k-\tau^h_{k-1})> \eps\right]
  &\le 
  \frac{N}{\eps^2}
  \left(
  2h+
2^{3/2}B h^\gamma
\right)^2=
 \frac{Nh^{2\gamma}}{\eps^2}
  \left(
  2h^{1-\gamma}+
2^{3/2}B
\right)^2\\
&
\le 
\frac{Th^{2\gamma-1}}{\eps^2}
  \left(
  2h^{1-\gamma}+
2^{3/2}B
\right)^2
.
  \end{split}
  \end{equation}
  Letting $h$ go to $0$ completes the proof.
 \end{proof}

\begin{propo}\label{prop:bound_below}
Suppose that Condition~(B) is satisfied. Then
 for all $y\in I^\circ$, $\eps\in (0,\infty)$ and $T\in [0,\infty)$ it holds that
 \begin{equation}\label{eq:ub_diff_ex_times}
 P_y\left [(H_l(Y^h)\wedge T)\le (H_l(Y)\wedge T)+\eps\right]\to 1, \quad h\to 0.
\end{equation} 
\end{propo} 
 \begin{proof}
 Throughout the proof fix $y\in I^\circ$, $\eps\in (0,\infty)$ and $T\in [0,\infty)$. For all $h\in (0,\ol h)$ we introduce the events
 $$
 A_h=\left\{\sup_{k\in \{1,\ldots,\lfloor (T+\eps)/h\rfloor\}}(k(h\wedge \alpha(h))-\tau^h_k)\le \frac{\eps}{3}\right\}\quad \text{and} \quad
 B_h=\left\{ \sup_{k\in \{1,\ldots,\lfloor (T+\eps)/h\rfloor\}}(\tau^h_k-\tau^h_{k-1})\le \frac{\eps}{3}\right\}.
  $$
  We show that for all $h\in (0,\ol h)$ small enough it holds that
  $\{(H_l(Y^h)\wedge T)-(H_l(Y)\wedge T)\le \eps\}\supseteq A_h\cap B_h$. Then Lemma~\ref{lem:ub_temp_diff_prob} and Lemma~\ref{lem:ub_time_lag} imply \eqref{eq:ub_diff_ex_times}.

First notice that for all $h\in (0,\ol h)$ on the event
$A_h\cap B_h \cap \{H_l(Y)\in(T,\infty]\}$
we have
\begin{equation}\label{eq:diff_exit_times_aux2}
(H_l(Y^h)\wedge T)-(H_l(Y)\wedge T)\le T-T=0.
\end{equation}
Next, there exists $h_0\in (0,\ol h)$ such that for all $h\in (0,h_0)$ it holds that
\begin{equation}\label{eq:lb_alphah}
\frac{\alpha(h)}{h}\ge 1-\frac{\eps}{3(T+\eps)}.
\end{equation}
This implies that for all $h\in (0, h_0)$ and $k\in \{1,\ldots,\lfloor (T+\eps)/h\rfloor\}$ we have
$$
kh=k(h\wedge \alpha(h))+kh\left(1-\left(1\wedge \frac{\alpha(h)}{h}\right)\right)
\le k(h\wedge \alpha(h))+\frac{kh\eps}{3(T+\eps)}
\le k(h\wedge \alpha(h))+\frac{\eps}{3}.
$$
Note that for all $h\in (0, h_0)$ and $k\in \N$ on the event $\{H_l(Y)\in (\tau^h_{k-1},\tau^h_k]\cap[0,T]\}$
it holds that $Y_{\tau^h_{k-1}}>l$ and $Y_{\tau^h_{k}}=l$ and, therefore, it follows from the construction of $Y^h$ (cf.\ \eqref{eq:lin_int_pol}) that $H_l(Y^h)=kh$.
Consequently, for all $h\in (0, h_0)$ and $k\in \{1,\ldots,\lfloor (T+\eps)/h\rfloor\}$ 
  on the event $A_h\cap B_h \cap \{H_l(Y)\in (\tau^h_{k-1},\tau^h_k]\cap[0,T]\}$ we have
  \begin{equation}\label{eq:diff_exit_times_aux1}
  \begin{split}
  (H_l(Y^h)\wedge T)-(H_l(Y)\wedge T )&=
 ((kh)\wedge T)-H_l(Y)\le
 kh-\tau^h_{k-1}\\
 &\le k(h\wedge \alpha(h))+\frac{\eps}{3}-\tau^h_{k-1}
 \le \tau^h_k+\frac{2\eps}{3}-\tau^h_{k-1}\le \eps.
 \end{split}
  \end{equation}
Finally, for all $h\in (0,h_0\wedge( \eps/3))$
on the event $A_h\cap B_h \cap \{H_l(Y)\in (\tau^h_{\lfloor (T+\eps)/h\rfloor},\infty]\}$ we have $Y_{\tau^h_{\lfloor (T+\eps)/h\rfloor}}>l$ and
again by the construction of $Y^h$ (cf.\ \eqref{eq:lin_int_pol}) that
  $$
  H_l(Y^h)\ge \left(\left \lfloor \frac{T+\eps}{h}\right \rfloor+1\right)h
  \ge \left( \frac{T+\eps}{h}\right )h>T.
  $$ 
 Therefore, by \eqref{eq:lb_alphah} we have for all $h\in (0,h_0\wedge( \eps/3))$
on the event $A_h\cap B_h \cap \{H_l(Y)\in (\tau^h_{\lfloor (T+\eps)/h\rfloor},\infty]\}$
  \begin{equation*}
  \begin{split}
   H_l(Y)>\tau^h_{\lfloor (T+\eps)/h\rfloor}
  &\ge \left \lfloor \frac{T+\eps}{h}\right \rfloor (h\wedge \alpha(h))-\frac{\eps}{3}
  \ge \left( \frac{T+\eps}{h}-1\right)h\left(1\wedge \frac{\alpha(h)}{h}\right)-\frac{\eps}{3}\\
   &\ge \left( \frac{T+\eps}{h}-1\right)h\left(1-\frac{\eps}{3(T+\eps)}\right)-\frac{\eps}{3}
 \ge T+\eps-h-\frac{\eps}{3}-\frac{\eps}{3}
  >T
  \end{split}
  \end{equation*}
  and, consequently,
$(H_l(Y^h)\wedge T)-(H_l(Y)\wedge T )=0$.
Combining this with 
\eqref{eq:diff_exit_times_aux2}
and~\eqref{eq:diff_exit_times_aux1}
proves that for all $h\in (0,h_0\wedge( \eps/3))$ it holds that
  $\{(H_l(Y^h)\wedge T)-(H_l(Y)\wedge T)\le \eps\}\supseteq A_h\cap B_h$.
  This completes the proof.
 \end{proof}

Combining Proposition~\ref{prop:bound_above} and Proposition~\ref{prop:bound_below} 
shows that for all $y\in I^\circ$, $\eps\in (0,\infty)$ and $T\in [0,\infty)$ we have
$$
P_y\left [\left|(H_l(Y^h)\wedge T)-(H_l(Y)\wedge T)\right|\le \eps\right]\to 1, \quad h\to 0,
$$
which is~\eqref{eq:1802a6}. The proof of Theorem~\ref{th:main_new} is thus completed.

\begin{remark}
It is worth noting that, under Condition~(D),
we can apply Corollary~5.4 of~\cite{aku2018cointossing}
and conclude that
for all $T\in(0,\infty)$ and $y\in I^\circ$
it holds
$$
\sup_{k\in\{1,\ldots,\lfloor T/h\rfloor\}}
|\tau^h_k-kh|
\xrightarrow[]{P_y}0,\quad h\to0.
$$
This allows to simplify the above argumentation
and results in an easier proof of the claim of
Theorem~\ref{th:main_new} under Condition~(D).
The latter, however,
does not provide an answer to the open question
about the convergence of the exit times of the (weak)
Euler approximations for the CEV diffusion.
To answer that question,
we need the full strength of Theorem~\ref{th:main_new}
under Condition~(B) (see Section~\ref{sec:CEV}).
\end{remark}

\section{Condition (A) does not suffice}\label{sec:A_not_suff}
For simplicity we assume that $r=\infty$ throughout this section.
Let $l>-\infty$ be an accessible boundary point of $Y$. Let $(l_h)_{h\in (0, \ol h)}\subset (l,\infty)$ be the numbers defined at the beginning of Section~\ref{sec:main_res}. It holds that $l_h\searrow l$ as $h\searrow 0$. Let $(\ol l_h)_{h\in (0,\ol h)}\subset (l,\infty)$ be another family satisfying 
$\ol l_h>l_h$ for all $h\in (0,\ol h)$ and $\ol l_h\searrow l$ as $h\searrow 0$. Then define for all $h\in (0, \ol h)$ a scale factor by $a_h(y)=\wh a_h(y)$  for all $y\ge \ol l_h$ and by $a_h(y)=0$ for all $y\in [l,\ol l_h)$, where $\wh a_h$ is the EMCEL($h$) scale factor defined in Example~\ref{ex:25022019a1}.
Notice that $I_h=I^\circ=(l,\infty)$
for the scheme with scale factors $(a_h)_{h\in(0,\ol h)}$
due to the second term
on the right-hand side of~\eqref{eq:22022019a2},
hence Condition~(B) is not satisfied
(the integral in~\eqref{eq:01042019a1}
is zero whenever $y\in I^\circ$ is close to~$l$).

By \eqref{eq:22022019a3} it holds for every $h\in (0,\ol h)$ and $y\ge \ol l_h$ that
$$
\frac{1}{2}\int_{(y-a_h(y),y+a_h(y))} (a_h(y)-|z-y|)\,m(dz) = \frac{1}{2}\int_{(y-\wh a_h(y),y+\wh a_h(y))} (\wh a_h(y)-|z-y|)\,m(dz) = h.
$$
For every compact subset $K$ of $(l,\infty)$ there exists $h_0\in (0,\ol h)$ such that 
$K \subset (\ol l_h,\infty)$ for all $h\in (0,h_0)$. This implies for all $h\in (0,h_0)$ that
$$
\sup_{y\in K }  \left|
\frac{1}{2}\int_{(y-a_h(y),y+a_h(y))} (a_h(y)-|u-y|)\,m(du)
-h
\right|
=0
$$
and hence Condition~(A) is satisfied for this choice of scale factors.
Theorem~\ref{thm_main_coin} applies
and we thus have the weak convergence of $X^h$ to $Y$
as $h\to0$.

Next, let $h\in (0,\ol h)$ and $y\ge \ol l_h$. This implies that $y>l_h$ and hence there exists
$\ol a \in (0,\infty)$ such that $y>l+\ol a$ and 
$$
\frac{1}{2}\int_{(l,l+2\ol a)} (\ol a-|u-(l+\ol a)|)\,m(du)\ge h.
$$
This implies that
\begin{equation}\label{eq:060319a1}
\frac{1}{2}\int_{(y-a_h(y),y+a_h(y))} (a_h(y)-|u-y|)\,m(du)=h
\le \frac{1}{2}\int_{(l,l+2\ol a)} (\ol a-|u-(l+\ol a)|)\,m(du).
\end{equation}
Now suppose that 
$y-a_h(y)=l$. Then we have for all $u\in (l+\ol a,y)$ that $\ol a-|u-(l+\ol a)|<a_h(y)-|u-y|$
and for all $u\in (l,l+\ol a)$ that $\ol a-|u-(l+\ol a)|=a_h(y)-|u-y|$. Hence it holds that
$$\frac{1}{2}\int_{(l,l+2\ol a)} (\ol a-|u-(l+\ol a)|)\,m(du)<\frac{1}{2}\int_{(y-a_h(y),y+a_h(y))} (a_h(y)-|u-y|)\,m(du),$$
which contradicts \eqref{eq:060319a1}. Hence, it must hold that $y-a_h(y)>l$.
To summarize, we have for all $y\in [\ol l_h,\infty)$ that $y-a_h(y)\in (l, y]$ and for all $y\in (l,\ol l_h)$ that $y\pm a_h(y)=y$. 
This implies that the discrete-time process $(X^h_{kh})_{k \in \bbN_0}$ (cf.\ \eqref{eq:def_X})
does not jump from the region $[\ol l_h,\infty)$ to the boundary point $l$. Moreover, once it enters the interval $(l,\ol l_h)$ it stays constant. Therefore, the process $(X^h_{kh})_{k \in \bbN_0}$ never hits the boundary point $l$
and it holds $H_l(X^h)=\infty$ for all $h\in (0,\ol h)$. In particular, we do not have weak convergence of $H_l(X^h)$ to $H_l(Y)$ and thus the claim of Theorem~\ref{th:main_new} does not hold true.

To sum up, for \emph{every} general diffusion $Y$
with state space $I=[l,\infty)$, $l>-\infty$ being an accessible boundary point,
we constructed a family of scale factors such that Condition~(A) is satisfied
but the claim of Theorem~\ref{th:main_new} does not hold true.
In terms of the boundary classification, an accessible boundary can be either \emph{regular} or \emph{exit}
(see, e.g., Definition~16.48 in~\cite{Breiman:92}).
It is worth noting that the above construction works irrespectively of whether $l$ is an exit boundary of $Y$ or a regular one.

\section{Applications}\label{sec:CEV}
\subsection{CEV diffusion}
We start with the CEV diffusion.
Let $p\in(-\infty,1)$.
Consider the process $Y$ driven by the SDE
\begin{equation}\label{eq:01042019a2}
dY_t=Y_t^p\,dW_t
\end{equation}
inside $I^\circ=(0,\infty)$
and stopped as soon as it reaches $0$.
Notice that the boundary point $0$
is accessible if and only if $p\in(-\infty,1)$
(a straightforward application of Feller's test for explosions;
see, e.g., Theorem~5.5.29 in~\cite{KS}),
hence such a restriction on the parameter~$p$.
The state space of $Y$ is thus $I=[0,\infty)$.
In other words, $Y$ is a diffusion in natural scale
with state space $I=[0,\infty)$
and speed measure
$m(dx)=\frac2{x^{2p}}\,dx$ on $I^\circ$.

In this subsection
we consider the question of approximating
the law of $H_0(Y)$.
As discussed in the introduction
this is a challenging problem
already in the case $p\in[\frac12,1)$
(see the text after~\eqref{eq:22032019a4} and,
for more detail, see~\cite{ChiganskyKlebaner:12}).
Of course, one possibility is to use the EMCEL scheme,
which works for every general diffusion,
let alone for~\eqref{eq:01042019a2}.
Below we discuss the weak Euler scheme
for~\eqref{eq:01042019a2}.

More precisely, we consider the weak Euler scheme
slightly modified near zero to exclude
the possibility of jumping out of the state space
(recall that the scale factors we consider
should satisfy $y\pm a_h(y)\in I$
whenever $y\in I^\circ$ and $h\in(0,\ol h)$).
Namely, let $\ol h=1$ and for all $h\in (0,1)$ let
\begin{equation}\label{eq:04042019a1}
a^{Eu}_h(y)=\min\{\sqrt h y^p,y\}=\begin{cases}
\sqrt h y^p&\text{if }y\in (h^{\frac{1}{2(1-p)}},\infty),\\
y&\text{if }y\in [0,h^{\frac{1}{2(1-p)}}].
\end{cases}
\end{equation}
For all $y\in I^\circ$ and $h\in (0,1)$ we denote by $X^{Eu,h,y}=(X^{Eu,h,y}_t)_{t\in [0,\infty)}$ the corresponding weak Euler scheme started in $y$ with step size $h$
defined via \eqref{eq:def_X} and~\eqref{eq:13112017a1}.

The main part of this subsection is devoted to the verification that
in the case $p\in(-\infty,1)\setminus(-\frac12,0)$
for the weak Euler scheme~\eqref{eq:04042019a1}
Condition~(B) is satisfied
(the remaining case $p\in(-\frac12,0)$
will be discussed thereafter).
To this end we need to distinguish the cases $p=\frac{1}{2}$
(Lemma~\ref{lem:cev_05})
and $p\in(-\infty,1)\setminus\left((-\frac12,0)\cup\{\frac12\}\right)$
(Lemma~\ref{lem:cev_neq_05}). 
Combining these results with
Theorem~\ref{th:main_new} leads
to the following corollary.

\begin{corollary}\label{cor:main_cev}
Let $p\in(-\infty,1)\setminus(-\frac12,0)$.
Then for any $y\in (0,\infty)$
the distributions of the random elements
$(H_0(X^{Eu,h,y}),X^{Eu,h,y})$
under $P$ converge weakly to the distribution of
$(H_0(Y),Y)$ under $P_y$, as $h\to 0$.
\end{corollary}

In particular, we have the weak convergence
$H_0(X^{Eu,h,y})\xrightarrow[]{w}H_0(Y)$
of the exit times.

\begin{lemma}\label{lem:cev_05}
In the case $p=\frac{1}{2}$
the scale factors $(a_h^{Eu})_{h\in (0,1)}$ satisfy Condition~(B).
\end{lemma}

\begin{proof}
In the calculations of this proof we use the convention that $0\log(0)=0$.
First note that for all
$y\in (0,\infty)$ and $a\in (0,y]$ it holds that
\begin{equation}\label{eq:int_m_cev}
\begin{split}
&\frac{1}{2}\int_{(y-a,y+a)} (a-|u-y|)\,m(du)
=\int_{y-a}^{y+a}\frac{a-|u-y|}{u}\,du\\
&=(y+a)\log\left(1+\frac{a}{y}\right)+(y-a)\log\left(1-\frac{a}{y}\right)\\
&=y\left[\left(1+\frac{a}{y}\right)\log\left(1+\frac{a}{y}\right)+\left(1-\frac{a}{y}\right)\log\left(1-\frac{a}{y}\right) \right].
\end{split}
\end{equation}
It follows that for all $a\in (0,\infty)$ we have
$$
\frac12\int_{(0, 2a)} (a - |u-a|) m(du)=
2\log(2)a
$$
and hence $l_h=\frac{h}{2\log(2)}$ for all $h\in (0,1)$. As $r=\infty$, it holds $r_h=\infty$ for all $h\in (0,1)$. Since $y-a^{Eu}_h(y)>0$ if and only if $y>h$, and since $l_h<h$,
it follows that $I_h=(l_h,\infty)$. 

Next, let the function $\varphi \colon (0,1)\times [0,\infty) \to \R$ be given by 
\begin{equation}\label{eq:def_varphi}
\varphi(h,y)=\frac{1}{2}\int_{(y-a^{Eu}_h(y),y+a^{Eu}_h(y))} (a^{Eu}_h(y)-|u-y|)\,m(du).
\end{equation}
It follows from \eqref{eq:04042019a1} and~\eqref{eq:int_m_cev} that
\begin{equation}\label{eq:04042019a2}
\varphi(h,y)=\begin{cases}
\psi\left(\frac hy\right)h&\text{if }y\in[h,\infty),\;h\in(0,1),\\
2\log(2)y&\text{if }y\in[0,h),\;h\in(0,1),
\end{cases}
\end{equation}
where the function $\psi\colon (0,1]\to (1,2\log(2)]$
is a continuous and strictly increasing bijection given by the formula
$$
\psi(x)=\frac{\left(1+\sqrt{x}\right)\log\left(1+\sqrt{x}\right)
+\left(1-\sqrt{x}\right)\log\left(1-\sqrt{x}\right)}{x}, \quad x\in (0,1].
$$

To verify Item~(i) of Condition~(B), we fix $h\in(0,1)$,
recall that $I_h=(\frac h{2\log(2)},\infty)$
and conclude from~\eqref{eq:04042019a2}
and the fact that $\psi$ is $(1,2\log(2)]$-valued that
$$
h=\inf_{y\in I_h}\varphi(h,y)\le\sup_{y\in I_h}\varphi(h,y)=2\log(2)h,
$$
i.e., \eqref{eq:01042019a1} is satisfied with $\alpha(h)=h$, $B=2\log(2)$ and $\gamma=1$.

To verify Item~(ii) of Condition~(B), let $K$ be a compact subset of $(0,\infty)$. 
The fact
 that $\psi$ is continuous and increasing 
 ensures for all $h\in (0,\min K)$ that
 $$
 \sup_{y\in K}\varphi(h,y)=\sup_{y\in K}\psi\left(\frac{h}{y}\right)h
 =\psi\left(\frac{h}{\min K}\right)h
 $$
 and hence it follows from the fact that
$\lim_{x\to 0} \psi(x)=1$ that 
$$
\lim_{h\to 0}\frac{\sup_{y\in K}\varphi(h,y)}{h}=\lim_{h\to 0}\psi\left(\frac{h}{\min K}\right)=1.
$$
This yields~\eqref{eq:04042019b1} and concludes the proof.
\end{proof}
 
\begin{lemma}\label{lem:cev_neq_05}
In the case $p\in(-\infty,1)\setminus\left((-\frac12,0)\cup\{\frac12\}\right)$
the scale factors $(a_h^{Eu})_{h\in (0,1)}$ satisfy Condition~(B).
\end{lemma}

\begin{proof}
 The arguments go along the lines of the proof of Lemma~\ref{lem:cev_05}. 
 Therefore, we only list the quantitative changes. 
 First note that for all
$y\in (0,\infty)$ and $a\in (0,y]$ it holds that
\begin{equation}\label{eq:int_m_cev_neq_05}
\begin{split}
&\frac{1}{2}\int_{(y-a,y+a)} (a-|u-y|)\,m(du)
=\int_{y-a}^{y+a}\frac{a-|u-y|}{u^{2p}}\,du\\
&=\frac{1}{2(2p-1)(1-p)}\left[
2y^{2(1-p)}-(y+a)^{2(1-p)}-(y-a)^{2(1-p)}
\right]\\
&=\frac{y^{2(1-p)}}{2(2p-1)(1-p)}\left[
2-\left(1+\frac{a}{y}\right)^{2(1-p)}-\left(1-\frac{a}{y}\right)^{2(1-p)}
\right].
\end{split}
\end{equation}
It follows that for all $a\in (0,\infty)$ we have
$$
\frac12\int_{(0, 2a)} (a - |u-a|) m(du)=
\frac{
1-2^{1-2p}}{(2p-1)(1-p)}a^{2(1-p)}
$$
and hence for all $h\in (0,1)$
$$
l_h=\left(\frac{(2p-1)(1-p)}{1-2^{1-2p}}h \right)^{\frac{1}{2(1-p)}}.
$$
As $r=\infty$, it holds $r_h=\infty$ for all $h\in (0,1)$.
The preceding calculations hold true, in fact,
for all $p\in(-\infty,1)\setminus\{\frac12\}$.
Given that $p\in(-\infty,1)\setminus\{\frac12\}$,
we have the following equivalence:
\begin{equation}\label{eq:06042019a1}
\frac{(2p-1)(1-p)}{1-2^{1-2p}}\le1
\end{equation}
holds if and only if
$p\in(-\infty,1)\setminus\left((-\frac12,0)\cup\{\frac12\}\right)$
(notice that the equality in~\eqref{eq:06042019a1}
holds true if and only if $p\in\{-\frac12,0\}$).
Hence, under the assumption of Lemma~\ref{lem:cev_neq_05},
$l_h\le h^{\frac1{2(1-p)}}$ for all $h\in(0,1)$.
Since $y-a^{Eu}_h(y)>0$ if and only if $y>h^{\frac{1}{2(1-p)}}$, we obtain that $I_h=(l_h,\infty)$. 

We again define the function $\varphi \colon (0,1)\times [0,\infty) \to \R$
by~\eqref{eq:def_varphi}
and infer from \eqref{eq:04042019a1} and~\eqref{eq:int_m_cev_neq_05}
that
$$
\varphi(h,y)=\begin{cases}
\psi\left(\frac{h^{\frac1{2(1-p)}}}y\right)h&\text{if }y\in\left[h^{\frac1{2(1-p)}},\infty\right),\;h\in(0,1),\\
\frac{1-2^{1-2p}}{(2p-1)(1-p)}y^{2(1-p)}&\text{if }y\in\left[0,h^{\frac1{2(1-p)}}\right),\;h\in(0,1),
\end{cases}
$$
where,
for $p\in(-\infty,1)\setminus\left([-\frac12,0]\cup\{\frac12\}\right)$,
the function
$\psi\colon(0,1]\to\left(1,\frac{1-2^{1-2p}}{(2p-1)(1-p)}\right]$
is a continuous and strictly increasing bijection given by the formula
\begin{equation}\label{eq:06042019a2}
\psi(x)=\frac{2-(1+x^{1-p})^{2(1-p)}-(1-x^{1-p})^{2(1-p)}}{2(2p-1)(1-p)x^{2(1-p)}}, \quad x\in (0,1],
\end{equation}
while, for $p\in\{-\frac12,0\}$, the function
$\psi\colon(0,1]\to\{1\}$ is identically~$1$
(but~\eqref{eq:06042019a2} still applies).
Now Condition~(B) is verified in qualitatively the same way
as it was done in the proof of Lemma~\ref{lem:cev_05}.
\end{proof}

\paragraph{Alternative approach for $p\in(-\infty,\frac12)$.}
As Corollary~\ref{cor:main_cev}
covers the case $p\in(-\infty,1)\setminus(-\frac12,0)$,
what follows is of particular interest
for $p\in(-\frac12,0)$,
although it applies
more generally
for $p\in(-\infty,\frac12)$.
The idea is to extend the state space beyond $0$
to make $0$ an interior point and get weakly
converging approximations for $H_0(Y)$
using any scheme $X^h$ satisfying
$X^h\xrightarrow[]{w}Y$, $h\to0$,
on the extended space (recall Remark~\ref{rem:01042019a1}).
It is worth noting that,
on the contrary, in the case $p\in[\frac12,1)$
we cannot extend the state space beyond $0$,
as $m((0,\eps))=\infty$ for any $\eps>0$ in the latter case,
so that \eqref{eq:06072018a1} cannot be satisfied
on any state space, where zero is an interior point.
For completeness, we now describe the approach
more precisely.

Below we assume as announced that $p\in(-\infty,\frac12)$.
Consider a continuous strong Markov process
$(\wt\Omega,\wt\cF,(\wt\cF_t)_{t\ge0},
(\wt P_y)_{y\in\bbR},(\wt Y_t)_{t\ge0})$
with state space $\bbR$
which is a weak solution to the SDE
\begin{equation}\label{eq:06042019a3}
d\wt Y_t=\eta(\wt Y_t)\,d\wt W_t
\end{equation}
with
\begin{equation}\label{eq:06042019a4}
\eta(x)=|x|^p+1_{\{0\}}(x),\quad x\in\bbR.
\end{equation}
Notice that~\eqref{eq:06042019a3}
has a unique in law weak solution because
the function $\eta$ satisfies the Engelbert-Schmidt conditions
\eqref{eq:27092018a2} and~\eqref{eq:27092018a3} on $\R$.
Put differently, $\wt Y$ is a diffusion in natural scale
with state space $\bbR$ and speed measure
$\wt m(dx)=\frac2{|x|^{2p}}1_{\bbR\setminus\{0\}}(x)\,dx$.

Fix an arbitrary starting point $y\in(0,\infty)$ and notice that
$$
\Law_{P_y}
\big(Y_{t}; t\in [0,\infty)\big)
=\Law_{\wt P_y}
\big(\wt Y_{t\wedge H_0(\wt Y)}; t\in [0,\infty)\big).
$$
In particular, the laws of the exit times
$H_0(Y)$ and $H_0(\wt Y)$ coincide.
As the functional $H_{0}\colon C([0,\infty),\R)\to [0,\infty]$
is $\wt P_y$-a.s.\ continuous
($0$ is an interior point of the state space $\bbR$ for~$\wt Y$),
we get the convergence
$H_0(\wt X^{h,y})\xrightarrow[]{w}H_0(\wt Y)$,
$h\to0$,
and hence the required convergence
$$
\Law_P\big(H_0(\wt X^{h,y})\big)
\xrightarrow[]{w}
\Law_{P_y}\big(H_0(Y)\big),
\quad h\to0,
$$
with any scheme $\wt X^{h,y}$ starting in $y$ satisfying
\begin{equation}\label{eq:06042019a5}
\Law_P\big(\wt X^{h,y}_t; t\in[0,\infty)\big)
\xrightarrow[]{w}
\Law_{\wt P_y}\big(\wt Y_t; t\in[0,\infty)\big),
\quad h\to0.
\end{equation}
It remains only to discuss for which schemes
we have~\eqref{eq:06042019a5}.
A universal possibility is to use
the EMCEL scheme for the SDE~\eqref{eq:06042019a3},
which ensures~\eqref{eq:06042019a5}
(recall that the EMCEL scheme works for
\emph{every} general diffusion).
Concerning other schemes
for which~\eqref{eq:06042019a5}
might hold true,
the following list contains
what is currently known from the literature
(to the best of our knowledge):
\begin{enumerate}[(1)]
\item
\cite{gyongy98} proves the almost sure convergence
of the Euler scheme for a diffusion
under the local Lipschitz condition
on the diffusion coefficient inside the domain.
This result does not apply,
as $\eta$ is not locally Lipschitz near the interior point~$0$.

\item
To mention a related result for
not locally Lipschitz diffusion coefficients,
\cite{AKL:11} proves the weak convergence
of the type~\eqref{eq:06042019a5}
for the Euler scheme for the CEV
diffusion~\eqref{eq:01042019a2}
(to be precise, not for~\eqref{eq:06042019a3}).
But they treat only the case $p\in[\frac12,1)$,
while we are considering $p\in(-\infty,\frac12)$ here.
Moreover, even if this result was available
also for $p\in(-\infty,\frac12)$,
we could not infer the weak convergence of the exit times
because the path functional $H_0$
is essentially discontinuous under $P_y$,
although it is essentially continuous under~$\wt P_y$.
In other words, we would then need such a result
for~\eqref{eq:06042019a3}
(not for~\eqref{eq:01042019a2}).

\item
Theorem~2.2 in~\cite{yan} establishes the weak convergence
of the Euler scheme for certain discontinuous,
hence not locally Lipschitz, diffusion coefficients $\eta$,
provided uniqueness in law holds for the SDE
(which we know for~\eqref{eq:06042019a3})
and $\eta$ is locally bounded and has at most linear growth.
This gives~\eqref{eq:06042019a5}
for the Euler scheme $\wt X^{h,y}$
for $\wt Y$ of~\eqref{eq:06042019a3}
whenever $p\in[0,1]$
(for us, the case of interest is $p\in[0,\frac12)$).
Alternatively, the latter statement can be inferred
from the results of the recent article~\cite{taguchi2019euler},
where again $\eta$ has to be locally bounded
and of at most linear growth.
Notice that for $p\in(-\infty,0)$ the mentioned results are not applicable
(local boundedness near $0$ is violated).

\item
In the case $p\in(-\infty,0)$
the convergence of the Euler scheme
for~\eqref{eq:06042019a3} is, indeed,
a delicate issue:
see the discussion in the end of Example~5.4
in~\cite{aku-jmaa} for the proof
that in the case $p=-1$
the law of the Euler scheme for~\eqref{eq:06042019a3}
does not weakly converge to the law of~$\wt Y$.

\item
We summarize the previous discussion by concluding
that in the case $p\in[0,\frac12)$
both the Euler and the EMCEL schemes
ensure~\eqref{eq:06042019a5},
while in the case $p\in(-\infty,0)$
the EMCEL scheme is the only one
currently known from the literature
for which \eqref{eq:06042019a5} is proved.
\end{enumerate}

\subsection{Squared Bessel processes}
As approximating hitting times of Bessel processes
is of special interest
(see \cite{DH:13} and~\cite{DH:17}),
in the end we briefly recall the relation between
the CEV diffusion and the squared Bessel process
(stopped at the time when it hits zero),
which yields a way to approximate the time
when a Bessel process hits zero.

Let $\delta\in(-\infty,2)$.
We consider the process $Z$ driven by the SDE
\begin{equation}\label{eq:01042019a3}
dZ_t=\delta\,dt+2\sqrt{Z_t}\,dW_t
\end{equation}
inside $I^\circ=(0,\infty)$ and stopped as soon as it hits
zero. Zero is accessible if and only if $\delta\in(-\infty,2)$.
Notice that, for $\delta\in(0,2)$,
$Z$ is the squared Bessel process of dimension $\delta$
stopped as it reaches zero.
Without stopping the solution to~\eqref{eq:01042019a3}
would be instantaneously reflecting at zero
whenever $\delta\in(0,2)$.
Recall that stopping when reaching a boundary point
is without loss of generality for our purposes
($H_0(Z)$ does not change).
Such a diffusion $Z$ is not in natural scale.
We define the process
$$
Y_t=s(Z_t),\quad t\in[0,\infty),
$$
with
$$
s(z)=(2-\delta)^{\delta-2} z^{1-\frac\delta2},
\quad z\in[0,\infty)
$$
(this is a variant of the scale function of~$Z$),
and conclude via It\^o's formula that $Y$ is exactly the CEV diffusion of~\eqref{eq:01042019a2}
with $p=1-\frac1{2-\delta}$.
Formally,
It\^o's formula is applied on the stochastic interval
$[0,H_0(Z))$ because $s$ has infinite derivative at zero.
This is enough for us because both $Z_t$ and $Y_t$
hit zero as $t\nearrow H_0(Z)$,
hence $H_0(Z)=H_0(Y)$,
and we indeed get the dynamics~\eqref{eq:01042019a2}
for $Y$ inside $(0,\infty)$
and $Y$ is stopped upon reaching zero.
As $H_0(Z)=H_0(Y)$,
we can approximate $H_0(Z)$
by approximating $H_0(Y)$,
e.g., as discussed in the previous subsection.

\paragraph{Acknowledgement}
We thank Stefan Ankirchner for an interesting discussion that initiated this research
and for pointing out several related references.
We acknowledge the support from the
\emph{German Research Foundation}
through the project 415705084.

\bibliographystyle{abbrv}
\bibliography{literature}

\end{document}